\documentclass{amsart}

\usepackage[usenames,dvipsnames]{xcolor}
\usepackage{amsmath, amsfonts, amssymb, amsthm, geometry, graphics, graphicx, multicol, multirow, enumerate, float, floatflt, fullpage, tikz}

\newcommand{\C}{\mathbb{C}}

\newcommand{\A}{\textnormal{\textbf{A}}}
\newcommand{\B}{\textnormal{\textbf{B}}}

\newtheorem{theorem}{Theorem}[section]
\newtheorem{proposition}[theorem]{Proposition}
\newtheorem{cor}[theorem]{Corollary}

\newtheorem{lemma}[theorem]{Lemma}

\title{Classification of Complex Cyclic Leibniz Algebras}
\author{Daniel Scofield and S. McKay Sullivan}
\address{Department of Mathematics, North Carolina State University, Raleigh, NC 27695}
\email{dscofie@ncsu.edu, smsulli4@ncsu.edu}

\begin{document}

\begin{abstract}
 Since Leibniz algebras were introduced by Loday as a generalization
 of Lie algebras, there has been a lot of interest in which results
 of the latter extend to the former.  Cyclic Leibniz algebras, those
 generated by one element, are a useful tool for this purpose.  In fact, they have no Lie algebra counterpart. Their
 simple structure lends itself to elegant counterexamples to the
 extension of several important results from Lie algebras to Leibniz algebras.  In this
 paper, we give a classification of complex cyclic Leibniz algebras.
\end{abstract}

\let\thefootnote\relax\footnote{The work of the authors was supported by NSF grant DMS-0943855.}

\maketitle
 \section{Introduction}
 Cyclic Leibniz algebras were introduced in \cite{Ray:2012} and appear
 in the classification of elementary \cite{Batten:2013} and minimal
 non-elementary Leibniz algebras \cite{Ray:2013}.  They are also used
 as examples in the expository article \cite{Demir:2013}.  In this
 work we classify these algebras in the complex case.  Good references are
 \cite{Ayupov:1998}, \cite{Barnes:2011}, \cite{Demir:2013} , and \cite{Gorbat:2013}.
 
 We recall that a Leibniz algebra is an algebra in which left
 multiplication by every element is a derivation, i.e., multiplication
 satisfies $x(yz) = (xy)z+y(xz)$
 for all $x,y,z \in \A$.  Note that with the further constraint $xy
  = -yx$ this becomes the definition of a Lie algebra.  

One major difference between Leibniz and Lie algebras is that
the product of an
element with itself in a Leibniz algebra may not be zero. 
Thus it makes sense to speak of Leibniz algebras generated by a single
element. Such algebras are called cyclic Leibniz
algebras. Several interesting results about these algebras have already
been obtained. For instance, cyclic Leibniz algebras have a unique Cartan subalgebra, they
have only finitely
many maximal subalgebras, and all of these subalgebras can be explicitly
computed \cite{Ray:2012}.

\section{Basic Structure}
\label{section2}
Let $\A$ be an $n$-dimensional vector space over $\C$
containing a nonzero element $a$. Choose a linear operator $T:\A \to \A$
such that $a$ is a cyclic vector for $T$, i.e., such that
$
\mathcal{B} =\{ a,T(a),\ldots,T^{n-1}(a)\}
$
is a basis for $\A$. Then 
$
T^n(a) = \alpha_1 a + \alpha_2 T(a)+ \cdots + \alpha_n T^{n-1}(a)
$
for some $\alpha_1,\cdots,\alpha_n \in \C$. We define a product $\C a
\times \A \to \A$ as follows:
$(ca)v = cT(v)$ for all $ v \in \A$ and $c \in \C$,
i.e., such that $T$ is left multiplication by $a$. Throughout the rest of the
paper we will adobt the notation $L_a$ when referring to this $T$.  To
avoid always writing the basis elements of $\A$ in terms of $L_a$, we
let $a^k=L_a^{k-1}(a)$.
 We
aim to extend this product linearly to all of $\A \times 
\A$ in such a
way that left multiplication is a derivation, or in other words, such
that $\A$ is a Leibniz algebra. 

\begin{proposition}
\label{multder}
In the setting defined above, left multiplication is a derivation
($\A$ is a Leibniz algebra) if and only if $L_{a^2}=0$.
\end{proposition}
\begin{proof}
Assume that left multiplication is a derivation. Then it is easy to
check that  $L_{a^2} = [L_a,L_a] = 0$, where $[\cdot,\cdot]$ is the
commutator bracket of the Lie algebra Der$(A)$.

Now assume $L_{a^2} = 0$. By definition $L_a^j(x) = L_a(L_a^{j-1}(x))$
for all $j\geq 2$ and for all $x \in \A$. So by induction $L_a^j =
0$ for all $j\geq 2$. Now let $x = c_1 a + c_2
a^2 + \cdots + c_n a^n \in \A$. Then by linearity and the fact
that $L_{a^j} = 0$ for all
$j \geq 2$, we have $L_x = c_1 L_a$. Thus it is enough to show that $L_a$ is a
derivation on $\A$. Since $\A$ has basis $\{a, a^2, \ldots, a^n \}$, we
need only check that 
\begin{equation}
\label{deronA}
L_a(a^ia^j) = L_a(a^i)a^j + a^iL_a(a^j) \ \ \ \ \text{ for all } 1
\leq i,j, \leq n.
\end{equation}
Both sides of (\ref{deronA}) are zero when $i>1$ and $a^{j+2}$ when
$i=1$. Thus $L_a$ is a derivation. 
\end{proof}

We will make use of the following consequence of  Proposition
\ref{multder}: if $\A$ is Leibniz,
then
 $
 0=a^{n+1}a = (\alpha_1a + \alpha_2 a^2 + \cdots+ \alpha_na^n)a = \alpha_1a^2.
 $
 Thus $\alpha_1 = 0$.

Of course, all of what we have said so far is well known. We could have just started by saying that a cyclic Leibniz algebra is
an algebra generated by a single element $a$ such that left
multiplication by $a$ is a derivation. We chose to include
the preceding details because the underlying vector space structure is
the heart of the following classification.
\section{Classification of Cyclic Leibniz Algebras over $\C$}

Let $\A$ be an $n$-dimensional cyclic Leibniz algebra over $\C$  generated by a single element
$a$.  If $aa^n = 0$, then $\A$ is the nilpotent cyclic Leibniz algebra
of dimension $n$ of which there is only one up to isomorphism.  Throughout the rest of this paper, we consider only non-nilpotent
cyclic Leibniz algebras, i.e., cyclic Leibniz algebras  where the generator $a$ satisfies
\begin{equation}
\label{multfora}
aa^n = \alpha_k a^k + \alpha_{k+1}a^{k+1}+ \cdots + \alpha_n a^n
\end{equation}
for some $2 \leq k \leq n$ and $\alpha_k \neq 0$.

From the discussion in the previous section, it is clear that
any choice of $\alpha_2,\ldots,\alpha_n$ defines a  cyclic Leibniz
algebra. However, differing choices of these
coefficients do not always yield non-isomorphic algebras.  A simple example is when $\A$ is 2-dimensional and
$aa^2=\alpha a^2$ with $\alpha \neq 0,1$. Let
$x=\frac{1}{\alpha}a$. Then clearly $x$ is a cyclic generator for $\A$
and 
\begin{equation}
\label{example1}
xx^2 = \frac{1}{\alpha^3}aa^2 = \frac{1}{\alpha^2}a^2 = x^2 \neq
\alpha x^2.
\end{equation}
Thus $\A$ itself has generators whose multiplications are
different.

We consider the question of when two cyclic Leibniz algebras of
the same dimension are isomorphic.
\begin{lemma}
\label{isomorphism}
Let $\A$ and $\B$ be two cyclic Leibniz algebras of dimension $n$.  Assume
$\A$ has a cyclic generator $a$ which satisfies
$
aa^n = \alpha_k a^k + \alpha_{k+1}a^{k+1} +\cdots+ \alpha_n a^n.
$
Then $\A$ is isomorphic to $\B$ if and only if $\B$ has a
cyclic generator $b$ which satisfies
$
%\label{multforb}
bb^n = \alpha_k b^k + \alpha_{k+1}b^{k+1} +\cdots+ \alpha_n b^n.
$
\end{lemma}
\begin{proof}
Suppose there is an isomorphism of
algebras $f:\A \to \B$. Then $b=f(a)$ satisfies $bb^n =
\alpha_kb^k + \alpha_{k+1}b^{k+1} + \cdots + \alpha_n b^n$.  Clearly $b$ generates $B$, since $f$ is also a vector space isomorphism.

For the other implication, suppose $\B$ has a generator $b$ with the
above multiplication. Let $f:\A \rightarrow \B$ be the vector space
isomorphism that sends $a^i \mapsto b^i$ for $1 \leq i \leq n$. To
show that this is a homomorphism of Leibniz algebras, we find that
$f(aa^i) = f(a^{i+1}) = b^{i+1} = bb^i = f(a)f(a^i)$ for $1 \leq i <
n$ and $f(aa^n) = f(\alpha_k a^k + \alpha_{k+1}a^{k+1} +\cdots+
\alpha_n a^n) = \alpha_k f(a^k) + \alpha_{k+1}f(a^{k+1}) +\cdots+
\alpha_n f(a^n) = \alpha_k b^k + \alpha_{k+1}b^{k+1} +\cdots+ \alpha_n
b^n = bb^n = f(a)f(a^n)$. From this we see that $f$ respects all
non-zero products in $\A$. Thus $f$ is an isomorphism of Leibniz algebras.
\end{proof}

Given a cyclic Leibniz algebra $\A$ of dimension $n$ with a generator
$a$ satisfying (\ref{multfora}), we aim to find the isomorphism class
of $\A$. By Lemma \ref{isomorphism}, it is enough to find all possible coefficients
 $\gamma_2,\ldots,\gamma_n \in \C$ such that there exists a generator
 $x \in \A$ satisfying
$
xx^n = \gamma_2 x^2 + \gamma_3x^3 + \cdots + \gamma_n x^n.
$

% In the basis $\mathcal{B} = \{a,a^2,\ldots,a^n \}$, the linear
% operator $L_a$ has matrix\\
% \[
% L_a = \left[\begin{array}{ccccccc}
% 0 & 0 & \cdots& 0 & \cdots & 0 & 0\\
% 1 & 0 & \cdots& 0 & \cdots & 0 & 0 \\
% 0 & 1 & \cdots& 0 &\cdots & 0 & 0 \\
% \vdots & \vdots & \ddots& \vdots & \ddots & \vdots & \vdots\\
% 0 & 0 & \cdots & 1 & \cdots &0 & \alpha_k\\
% \vdots & \vdots & \ddots & \vdots & \ddots & \vdots & \vdots\\
% 0 & 0 & \cdots & 0 & \cdots & 1 & \alpha_n\\
% \end{array}\right],
% \]
% \newline
% which is the companion matrix for the 

Since $a$ is a cyclic vector for $L_a$, it follows that $L_a$ has characteristic polynomial
$
f(t) = t^n-\alpha_nt^{n-1}-\alpha_{n-1}t^{n-2}- \cdots - \alpha_k t^{k-1}.
$
By the Cayley-Hamilton theorem $f(L_a) = 0$.  In other
words
\begin{equation}
\label{cayley}
(L_a^n-\alpha_n L_a^{n-1}-\alpha_{n-1}L_a^{n-2}- \cdots - \alpha_k
L_a^{k-1})(x) = 0
\end{equation}
for all $x \in \A$. 
Now let us assume $x$ is a cyclic generator and write $x$ in terms of the
basis $\mathcal{B}$:
 $
x = c_1a + c_2 a^2 + \cdots + c_n a^n.
$
By rearranging
(\ref{cayley}) we obtain
\begin{equation*}
%\label{mult}
L_a^n(x) = \alpha_kL_a^{k-1}(x)+ \alpha_{k+1}L_a^{k}(x)+\cdots +
\alpha_n L_a^{n-1}(x).
\end{equation*}
Note that $c_1 \neq 0$ else $x$ is not a cyclic generator for $\A$. We
multiply by $c_1^n$ which gives
\[
c_1^nL_a^n(x) = c_1^n\alpha_kL_a^{k-1}(x)+ c_1^n\alpha_{k+1}L_a^{k}(x)+\cdots +
c_1^n\alpha_n L_a^{n-1}(x).
\]
From the proof of Proposition \ref{multder}, we know that $L_x
= c_1 L_a$.  Then
\[
L_x^n(x) = c_1^{n-k+1}\alpha_kL_x^{k-1}(x)+ c_1^{n-k}\alpha_{k+1}L_x^{k}(x)+\cdots +
c_1\alpha_n L_x^{n-1}(x),
\]
which we may also write as
\begin{equation}
\label{multforx}
xx^n = c_1^{n-k+1}\alpha_kx^k+ c_1^{n-k}\alpha_{k+1}x^{k+1}+\cdots +
c_1\alpha_n x^n.
\end{equation}
Thus every generator for $\A$ satisfies (\ref{multforx}).  Since
$x=c_1 a$ is a generator for all $c_1 \neq 0$, there is at least one
cyclic generator for $\A$ satisfying (\ref{multforx}) for every $c_1 \neq 0$.
We have proven the following lemma:
\begin{lemma}
\label{multtype}
Let $\A$ be an $n$-dimensional cyclic Leibniz algebra with generator
$a$ satisfying \textnormal{(\ref{multfora})}.  Then 
\begin{enumerate}
\item[(i)] Every generator $x = c_1a + \cdots +c_na^n$ for $\A$ satisfies
\begin{equation}
\label{multforx2}
xx^n = c_1^{n-k+1}\alpha_kx^k+ c_1^{n-k}\alpha_{k+1}x^{k+1}+\cdots +
c_1\alpha_n x^n.
\end{equation}

\item[(ii)] $\A$ has at least one generator satisfying
  \textnormal{(\ref{multforx2})} for each $c_1 \neq 0$.
\end{enumerate}
\end{lemma}

We call an $n$-dimensional non-nilpotent cyclic Leibniz
algebra $\A$ a \textit{type $k$} cyclic Leibniz algebra if $\A$ has a generator $x$ with multiplication
\begin{equation}
\label{type}
xx^n = x^k + \gamma_{k+1}x^{k+1}+\cdots+\gamma_n x^n
\end{equation}
for some ordered $(n-k)$-tuple $(\gamma_{k+1},\ldots,\gamma_n) \in
\C^{n-k}$.  
\begin{lemma} 
\label{typek}
Every $n$-dimensional non-nilpotent cyclic Leibniz algebra is of type $k$ for one and
only one $\displaystyle k \in \{2,\ldots,n\}$.
\end{lemma}
\begin{proof}
Let $\A$ be a non-nilpotent cyclic Leibniz algebra having a generator
$a$ with multiplication as in (\ref{multfora}).  Let $x = c_1a$ where 
$\displaystyle
c_1 =\alpha_k^{\frac{1}{k-n-1}}.
$
Then by Lemma \ref{multtype} part \textit{(i)}, $x$ is a generator satisfying (\ref{type}) with 
$
\gamma_{k+i} = c_1^{n-k+1-i}\alpha_{k+i}
$.
Thus $\A$ is
of type $k$ for at least one $k$.  That this $k$ is unique again follows
immediately from part \textit{(i)} of Lemma \ref{multtype}, since
$c_1\neq 0$ and $\alpha_k \neq 0$ imply that $c_1^{n-k+1}\alpha_k \neq
0$.
\end{proof}
By Lemma \ref{typek}, we know that any non-nilpotent cyclic Leibniz
algebra of type $k$ has a generator satisfying (\ref{type}) for some
$(n-k)$-tuple $(\gamma_{k+1},\ldots,\gamma_n)$. The question remains as
to whether $\A$ can also have a generator satisfying (\ref{type}) for
some other $(n-k)$-tuple $(\gamma_{k+1}',\ldots,\gamma_n')$.

Let $d=n-k$. We define the following relation on $\C^{d}$.  We say
$(\gamma_1,\gamma_2,\ldots,\gamma_{d}) \sim (\gamma'_1,\gamma'_2,\ldots,\gamma'_{d})$ if
$
(\gamma_1,\gamma_2,\ldots,\gamma_{d}) =  (\omega^{d}\gamma'_1,\omega^{d-1} \gamma'_2,\ldots,\omega\gamma'_{d})
$
for some $(d+1)$-th root of unity $\omega$. One may easily check
that $\sim$ is an equivalence relation on $\C^{d}$.  Then the
equivalence classes are of the form
\[
[(\gamma_1,\gamma_2,\ldots,\gamma_{d})] =
\{(\omega^{d}\gamma_1,\omega^{d-1}\gamma_2,\ldots,\omega\gamma_d)\text{
  }|
\text{ }\omega \text{ is an }(d+1)\text{-th root of unity}\}.
\]
Then we have the following lemma.
\begin{lemma}
\label{similar}
 Fix $k \in \{2, \ldots, n \}$ and let
  $(\gamma_{k+1},\ldots,\gamma_n) \in (\C^{n-k}) \setminus 0$. Let $\A$
  be a cyclic Leibniz algebra of dimension $n$ containing a generator $x$
  such that
$
xx^n = x^k + \gamma_{k+1}x^{k+1}+ \cdots + \gamma_n x^n.
$
Then $\A$ also has a generator $y$ such that
$
yy^n = y^k + \gamma'_{k+1}y^{k+1} + \cdots +\gamma'_n y^n
$
if and only if $(\gamma'_{k+1},\ldots, \gamma'_n) \sim
(\gamma_{k+1},\ldots, \gamma_n),$ where $\sim$ is the equivalence
relation on $\C^{n-k}$ defined above.
\end{lemma}
\begin{proof}
Assume $y$ satisfies the equation given above and write $y = c_1x+c_2x^2 + \cdots + c_n x^n$.  Then Lemma
\ref{multtype} says that $c_1^{n-k+1} = 1$.  Then $c_1$ is an
$(n-k+1)$-th root of unity, and by Lemma \ref{multtype} part
\textit{(i)} we have
$
yy^n = y^k + c_1^{n-k}\gamma_{k+1}y^{k+1}+\cdots + c_1 \gamma_n y^n.
$
Thus $(\gamma_{k+1},\ldots,\gamma_n) \sim
(\gamma'_{k+1},\ldots,\gamma'_n).$

For the other implication, let $(\gamma'_{k+1},\ldots,\gamma'_n) \in
[(\gamma_{k+1},\ldots,\gamma_n)]$.  Then 
$
(\gamma'_{k+1},\ldots, \gamma'_n) = (\omega^{n-k}\gamma_{k+1},\ldots,\omega \gamma^n)
$
for some $(n-k+1)-th$ root of unity $\omega$.  Then the generator $y = \omega x$ satisfies
$
yy^n = y^k + \omega^{n-k}\gamma_{k+1}y^{k+1}+\cdots + \omega \gamma_n
y^n= y^k + \gamma'_{k+1}y^{k+1}+ \cdots + \gamma'_n y^n.
$
\end{proof}
We have shown that there is a one-to-one correspondence between the 
isomorphism classes of non-nilpotent $n$-dimensional cyclic Leibniz algebras of type
$k$ and the nonzero elements of $\C^{n-k}/\sim$. More precisely, we
have the following classification:
\begin{theorem}[Classification]
\label{classificationtheorem}
Let $\A$ be an $n$-dimensional cyclic Leibniz algebra over $\C$. Then
$\A$ is isomorphic to a Leibniz algebra spanned by
$\{a,a^2,\ldots,a^n\}$ with the product $aa^n$ given by one and only one of the following:
\begin{enumerate}
\item[(1)]{
$aa^n = 0$ (nilpotent case).
}
\item[(2)]{
$aa^n = a^n$.
}
\item[(3)]{$aa^n = a^k + \alpha_{k+1}a^{k+1}+\cdots + \alpha_na^n, \ \
    \ \ $ $2\leq k \leq n-1$, $ \ \ \ \ (\alpha_{k+1}, \ldots, \alpha_n) \in \C^{n-k}/\sim$.}
\end{enumerate}

\end{theorem}
\begin{proof}
That there is only one $n$-dimensional nilpotent cyclic Leibniz
algebra up to isomorphism follows from Lemma \ref{isomorphism}. Now
assume $\A$ is a non-nilpotent $n$-dimensional cyclic Leibniz
algebra. Then by Lemma \ref{typek}, $\A$ has a generator satisfying
one and only one of \textit{(2)} or \textit{(3)}. Now assume $\A$ has a
generator $a$ satisfying 
$
aa^n = a^k + \alpha_{k+1}a^{k+1}+\cdots + \alpha_na^n.
$
Then by Lemma \ref{similar}, $\A$ also has a generator $b$ satisfying 
$
aa^n = a^k + \alpha_{k+1}'a^{k+1}+\cdots + \alpha_n'a^n
$
if and only if $(\alpha'_{k+1},\ldots,\alpha'_n)\sim (\alpha_{k+1},\ldots,\alpha_n)$.
\end{proof}

We think it worth nothing that for each $k=2,\ldots, n-1$ there is an
$(n-k)$-parameter family of isomorphism classes of cyclic Leibniz algebras of type $k$, the
parameters being chosen from the uncountable set $\C^{n-k}/\sim$. Thus for $n\geq 3$ there are
uncountably many isomorphism classes of cyclic Leibniz algebras of
dimension $n$ over $\C$.

\section{3 and 4-dimensional Classification}
We use the 3 and 4-dimensional cases to illustrate Theorem
\ref{classificationtheorem}:
\begin{cor}[3-dimensional Classification]

Let $\A$ be a 3-dimensional cyclic Leibniz algebra over $\C$. Then
$\A$ is isomorphic to a Leibniz algebra spanned by $\{a,a^2,a^3 \}$
with the product $aa^3$ given by one and only one of the following:
\begin{enumerate}
\item{
$aa^3 = 0$ (nilpotent case).
}
\item{
$aa^3 = a^3$.
}
\item{
$aa^3 = a^2 + \alpha_3 a^3, \ \ \ \ \alpha_3 \in \C/\sim, \ \ $ where
$\alpha \sim \alpha'$ if $\alpha' = \pm \alpha$.
}
\end{enumerate}
\end{cor}
\begin{cor}[4-dimensional Classification]
Let $\A$ be a 4-dimensional cyclic Leibniz algebra over $\C$. Then
$\A$ is isomorphic to a Leibniz algebra spanned by $\{a,a^2,a^3,a^4
\}$ with the product $aa^4$ given by one and only one of the
following:
\begin{enumerate}
\item{
$aa^4 = 0$ (nilpotent case). 
}
\item{
$aa^4 = a^4$.
}
\item{
$aa^4 = a^3 + \alpha_4 a^4, \ \ \ \ \alpha_4 \in \C/\sim, \ \ $ where
$\alpha \sim \alpha'$ if $\alpha' = \pm \alpha$.
}
\item{
$aa^4 = a^2 + \alpha_3 a^3 + \alpha_4 a^4, \ \ (\alpha_3,\alpha_4) \in
\C/\sim, \ \ $ where $(\alpha,\beta) \sim (\alpha',\beta')$ if
\[
(\alpha',\beta') \in
\{(\alpha,\beta),(\omega^2\alpha,\omega\beta),(\omega\alpha,\omega^2\beta)
\text{ }|
 \text{ }\omega=e^{2\pi i /3} 
\}
\]
}
\end{enumerate}
\end{cor}

\section*{Acknowledgements}
The authors are graduate students at North Carolina State University. Their work was funded by an REG grant from the National Science Foundation. The authors would like to thank Dr. Ernest Stitzinger for his guidance and support.

\bibliography{mainbib}{}
\bibliographystyle{plain}

\end{document}